%% file: univtriv13.tex
\documentclass[12pt]{amsart}
\usepackage{amscd,verbatim}
\usepackage{amssymb,amsmath}
\usepackage[all]{xy}
\usepackage[colorlinks,linkcolor=blue,citecolor=blue,urlcolor=red,dvipdfm]{hyperref}
\usepackage{enumerate}
\usepackage{bm}
\usepackage{tikz}
\usetikzlibrary{calc}

\newcommand{\A}{\mathbb{A}}
\newcommand{\F}{\mathbb{F}}
\newcommand{\G}{\mathbb{G}}

\renewcommand{\P}{\mathbb{P}}
\newcommand{\Q}{\mathbb{Q}}
\newcommand{\Z}{\mathbb{Z}}

\newcommand{\sO}{\mathcal{O}}

\newcommand{\bP}{\mathbb{P}}

\newcommand{\Cor}{\operatorname{\mathbf{Cor}}}

\newcommand{\HI}{{\operatorname{\mathbf{HI}}}}
\newcommand{\PR}{{\operatorname{\mathbf{PRig}}}}
\newcommand{\PRig}{{\operatorname{\mathbf{PRig}}}}
\newcommand{\PI}{{\operatorname{\mathbf{PI}}}}

\newcommand{\Frac}{\operatorname{Frac}}

\newcommand{\PST}{{\operatorname{\mathbf{PST}}}}

\newcommand{\NST}{\operatorname{\mathbf{NST}}}

\newcommand{\DM}{\operatorname{\mathbf{DM}}}

\newcommand{\uHom}{\operatorname{\underline{Hom}}}

\newcommand{\Coker}{\operatorname{Coker}}

\newcommand{\Br}{\operatorname{Br}}
\newcommand{\Spec}{\operatorname{Spec}}
\newcommand{\Proj}{\operatorname{Proj}}
\newcommand{\Fld}{\operatorname{\mathbf{Fld}}}
\newcommand{\Sm}{\operatorname{\mathbf{Sm}}}
\newcommand{\Sch}{\operatorname{\mathbf{Sch}}}
\newcommand{\Ab}{\operatorname{\mathbf{Ab}}}

\newcommand{\tr}{{\operatorname{tr}}}

\newcommand{\eff}{{\operatorname{eff}}}

\newcommand{\op}{{\operatorname{op}}}

\newcommand{\red}{{\operatorname{red}}}

\newcommand{\Zar}{{\operatorname{Zar}}}
\newcommand{\Nis}{{\operatorname{Nis}}}
\newcommand{\et}{{\operatorname{\acute{e}t}}}

\newcommand{\id}{{\operatorname{Id}}}

\newcommand{\pd}{{\partial}}

\newcommand{\CH}{{\operatorname{CH}}}

\renewcommand{\lim}{\operatornamewithlimits{\varprojlim}}
\newcommand{\colim}{\operatornamewithlimits{\varinjlim}}

\newcommand{\ol}{\overline}

\renewcommand{\phi}{\varphi}
\renewcommand{\epsilon}{\varepsilon}

\newcommand{\gm}{{\operatorname{gm}}}
\newcommand{\ur}{{\operatorname{ur}}}

\newcommand{\trdeg}{\operatorname{\mathbf{trdeg}}}

\newcommand{\pr}{\operatorname{pr}}

\newcommand{\Deg}{\operatorname{deg}}

\newcommand{\Qbullet}{Q_\bullet (X)} 
\newcommand{\olDelta}[1]{{\ol \Delta}^{#1}} 
\newcommand{\olGamma}{\ol \Gamma} 
\newcommand{\tildegamma}{\widetilde \Gamma}
\newcommand{\br}[1]{\left\{ #1 \right\} }
\newcommand{\midd}{\ \middle|\ }
\newcommand{\downin}{\rotatebox{-90}{$\in $}}
\newcommand{\Uprime}{T} 

\newtheorem{lemma}{Lemma}[section]

\newtheorem{thm}[lemma]{Theorem}
\newtheorem{theorem}[lemma]{Theorem}
\newtheorem{prop}[lemma]{Proposition}
\newtheorem{proposition}[lemma]{Proposition}
\newtheorem{cor}[lemma]{Corollary}
\newtheorem{corollary}[lemma]{Corollary}

\theoremstyle{definition}

\newtheorem{definition}[lemma]{Definition}

\theoremstyle{remark}

\newtheorem{remark}[lemma]{Remark}

\numberwithin{equation}{section}

\newtheorem*{ack}{Acknowledgments}

\begin{document}

\title[Unramified logarithmic Hodge-Witt cohomology]
{Unramified logarithmic Hodge-Witt cohomology \\ and $\P^1$-invariance}

\author[W. Kai]{Wataru Kai}
\address{Institute of Mathematics\\ Tohoku University\\ Aoba\\ Sendai 980-8578\\ Japan}
\email{kaiw@tohoku.ac.jp}

\author[S. Otabe]{Shusuke Otabe}
\address{Department of Mathematics, School of Engineering, Tokyo Denki University, 5 Senju Asahi\\ Adachi\\ Tokyo 120-8551\\ Japan
}
\email{shusuke.otabe@mail.dendai.ac.jp}

\author[T. Yamazaki]{Takao Yamazaki}
\address{Institute of Mathematics\\ Tohoku University\\ Aoba\\ Sendai 980-8578\\ Japan}
\email{takao.yamazaki.b6@tohoku.ac.jp}

\date{\today}

 
\thanks{
The first author is supported by JSPS KAKENHI Grant (JP18K13382).    
The second author was supported by JSPS KAKENHI Grant (JP19J00366). 
The third author is supported by JSPS KAKENHI Grant (JP18K03232, JP21K03153). 
}

\begin{abstract}
Let $X$ be a smooth proper variety over a field $k$
and suppose that 
the degree map $\CH_0(X \otimes_k K) \to \Z$ is isomorphic for any field extension $K/k$.
We show that $G(\Spec k) \to G(X)$ is an isomorphism
for any $\P^1$-invariant Nisnevich sheaf with transfers $G$.
This generalizes a result of Binda--R\"ulling--Saito
that proves the same conclusion for reciprocity sheaves.
We also give a direct proof of the fact that
the unramified logarithmic Hodge-Witt cohomology
is a $\P^1$-invariant Nisnevich sheaf with transfers.
\end{abstract}

\maketitle

\section{Introduction}

A proper smooth variety $X$ over a field $k$ is said to be \textit{universally $\CH_0$-trivial} if for any field extension $K/k$, the degree map of the Chow group of zero-cycles induces an isomorphism $\Deg\colon\CH_0(X\otimes_kK)\xrightarrow{\simeq}\Z$. 
Basic examples of universally $\CH_0$-trivial varieties include rational (and more generally stably rational) varieties, and this property may be considered as a near rationality condition. 
The condition plays a crucial role in the degeneration method established by Voisin\,\cite{Voisin} and Colliot-Th\'el\`ene--Pirutka\,\cite{CTP}, where counterexamples to the L\"uroth problem are produced. 

Now it is natural to ask how to disprove the universal $\CH_0$-triviality for a given variety $X$. In this direction, Merkurjev (see \cite[Theorem 2.11]{M}) proved that $X$ is universally $\CH_0$-trivial if and only if the function field $k(X)$ has trivial unramified cohomology, i.e.\ $M_*(k)\simeq M_*(k(X))_{\ur}$ for all Rost's cycle modules $M_*$ over $k$. 
%
As a consequence, if $\ell$ is a prime number different from the characteristic $p$ of $k$, then $\ell$-primary torsion elements of the Brauer group 
$\Br(X):=H^2_{\et}(X,\G_m)$ not coming from $\Br(k)$
obstruct the universal $\CH_0$-triviality. This is because the $\ell$-primary torsion subgroups $\Br(K)[\ell^{\infty}]\simeq H^2_{\et}(K,\Q_{\ell}/\Z_{\ell}(1))$ for all field extensions $K/k$ give rise to a cycle module $M_*\colon K\mapsto\bigoplus_{i=0}^{\infty}H^{i+1}_{\et}(K,\Q_{\ell}/\Z_{\ell}(i))$, to which the theorem of Merkurjev can be applied.

There are, however, more obstructions other than cycle modules.
In \cite{Totaro16}, Totaro adopted as an obstruction the sheaves of differential forms $\Omega^i_{X/k}$ 
to disprove the universal $\CH_0$-triviality for a wide class of hypersurfaces.
%
%
In \cite{ABBG2}, Auel et al.\  used $\Br(-)[2^\infty]$ in characteristic $p=2$
to obtain a similar result for conic bundles over $\P^2$.
Neither $\Omega^i$ nor $\Br(-)[p^\infty]$ in characteristic $p>0$ constitute a cycle module.
In fact, it is not straightforward to extend Merkurjev's result to $\Br(-)[p^\infty]$. This gap was filled in by their previous work \cite{ABBG}:

\begin{theorem}{\rm (see \cite[Theorem 1.1]{ABBG})}\label{thm:Auel et al}
Let $X$ be a smooth proper variety over a field $k$
which is universally $\CH_0$-trivial.
Then the structure morphism $X\to\Spec k$ induces an isomorphism $\Br(k)\xrightarrow{\simeq}\Br(X)$.
\end{theorem}

Our main result extends Theorem \ref{thm:Auel et al} to more general invariants:


\begin{theorem}[see Corollary \ref{cor:Mer-Kah-gen}]\label{thm:intro}
Let $X$ be a smooth proper variety over a field $k$
which is universally $\CH_0$-trivial.
Then the structure morphism $X\to\Spec k$ induces an isomorphism $G(k)\xrightarrow{\simeq}G(X)$
for any $\P^1$-invariant Nisnevich sheaf with transfers $G$
in the sense of Definition \ref{def:PI}.
\end{theorem}

Theorem \ref{thm:intro} generalizes Theorem \ref{thm:Auel et al}, since the Brauer group has a structure of a $\bP^1$-invariant Nisnevich sheaf with transfers 
(see Remark \ref{rem:intro} \eqref{item:br} below).
The conclusion of Theorem \ref{thm:intro} for \textit{homotopy invariant} sheaves with transfers follows from Merkurjev's result cited above. More recently, Binda et al.\ proved the same conclusion for another class of Nisnevich sheaves with transfers, called \textit{reciprocity sheaves} \cite[Theorem 10.12, Remark 10.13]{BRS}.
Our main theorem also covers their result,
since we have 
\begin{equation}\label{eq:HI-Rec-PI}
\text{homotopy invariant}
\Rightarrow 
\text{reciprocity}
\Rightarrow 
\text{$\P^1$-invariant}
\end{equation}
(see \cite[Theorems 3, 8]{rec}).
Both implications are strict (see Remark \ref{rem:compare}). 
Note also that $\Omega^i$ is a reciprocity sheaf by \cite[Theorem A.6.2]{rec},
hence Totaro's method can be explained
either by the results of Binda et al.\ or ours.
The main technical issue in the proof of Theorem \ref{thm:intro}
is the comparison of $G(X)$ and $h^0(G)(X)$,
where $h^0(G)$ is the \textit{maximal homotopy invariant subsheaf with transfers} of $G$.
We rephrase the problem in terms of algebraic cycles, 
and settle it  by establishing a new moving lemma (Theorem \ref{thm:moving}).

The unramified logarithmic Hodge-Witt cohomology
$H^1_\ur(-, W_n \Omega^i_{\log})$ (see \S \ref{sec:unram-deRW} for the definition)
satisfies the hypothesis of Theorem \ref{thm:intro}.
Although this fact can also be deduced from known results on reciprocity sheaves (see Remark \ref{rem:use-rec}), we will give a direct proof which depends on  classical results \cite{GS, Izhboldin} but not on reciprocity sheaves.


\begin{proposition}[see Proposition \ref{prop:urcoh-logHW-PI}]\label{prop:intro}
The unramified logarithmic Hodge-Witt cohomology
$H^1_\ur(-, W_n \Omega^i_{\log})$
is a $\P^1$-invariant Nisnevich sheaf with transfers
(over a field of positive characteristic)
for any integers $n \ge 1$ and $i \ge 0$.
\end{proposition}

As a corollary, we obtain a new proof of the following (known) result. 

\begin{theorem}\label{mainthm}
Let $X$ be a smooth proper variety over a field $k$
of characteristic $p>0$.
Assume that $X$ is universally $\CH_0$-trivial. 
Then the canonical map
\[ H^1_\ur(\Spec k, W_n \Omega^i_{\log})
\to H^1_\ur(X, W_n \Omega^i_{\log})
\]
is an isomorphism for any integers $n \ge 1$ and $i \ge 0$.
\end{theorem}

\begin{remark}\label{rem:intro}
\begin{enumerate}
\item\label{item:br}
Since $H^1_{\ur}(X,W_n\Omega^1_{\log})\simeq\Br(X)[p^n]$,
Theorem \ref{mainthm} for $i=1$ follows from Theorem \ref{thm:Auel et al}.
For general $i$, Theorem \ref{mainthm} was 
posed as a problem by Auel et al. in \cite[Problem 1.2]{ABBG}
and previously proved in \cite{BRS} and \cite{O}, as explained below.
\item 
Theorem \ref{mainthm} was shown by Binda et al.\ in \cite[Theorem 10.12, Remark 10.13]{BRS}
as a consequence of their general result on reciprocity sheaves mentioned above,
along with the fact that $H^1_{\ur}(-,W_n\Omega^i_{\log})$ has a structure of reciprocity sheaf.
\item 
Independently of \cite{BRS}, almost at the same time, Otabe also obtained Theorem \ref{mainthm} for $n=1$ (see \cite[Theorem 1.2]{O}). His proof is somewhat close to ours, but it is more cycle module theoretical. A \textit{tame subgroup} of the unramified cohomology was used in place of $h^0(G)$ in the present paper. The relation between these two subgroups is left for future research.
%
%
\item 
A similar statement as Theorem \ref{mainthm} holds 
when  $H^1_\ur$ is replaced by $H^j_\ur$ for any $j \in \Z$.
Indeed, 
the cohomology groups in question are trivial unless $j = 0, 1$,
because
the natural map $H^j_\ur(X, W_n \Omega^i_{\log})
\to H^j_\ur(k(X), W_n \Omega^i_{\log})$
is injective (if $X$ is connected),
and the $p$-cohomology dimension of 
any field of characteristic $p>0$ is at most one.
The case $j=0$ follows from
the results of Bloch--Gabber--Kato \cite[Theorem 2.1]{BK86}
and Merkurjev \cite[Theorem 2.11]{M}. 
\end{enumerate}
\end{remark}

The organization of the present paper is as follows. 
In \S\ref{sec:HI}, we revisit the proof of Theorem \ref{thm:intro} for homotopy invariant Nisnevich sheaves with transfers  due to Merkurjev \cite{M} and Kahn \cite{K} (see Corollary \ref{cor:Mer-Kah}). 
In \S\ref{sec:PI}, we state the main result in a slightly more general form (see Theorem \ref{prop:reduction-moving}) and prove it while admitting the key moving lemma (Theorem \ref{thm:moving}). The proof of Theorem \ref{thm:moving} occupies the next two sections.

In \S\ref{sec:simplicial}, we rephrase the problem in terms of algebraic cycles. To do so, we consider the Suslin complex $C_{\bullet}(X)$ and its variant $\overline{C}_{\bullet}(X)$, where the latter is defined by replacing $\A^n$ with $\P^n$ in the former. Theorem \ref{thm:moving} is then reduced to a comparison, up to Zariski sheafification, of their $0$-th homology presheaves (see Theorem \ref{thm:moving2}).
Its proof is given in  \S\ref{sec:moving lemma},
which is pivotal in our work.
%
%

In \S\ref{sec:unram-deRW}, we give a proof of Proposition \ref{prop:intro}. 
Finally, Appendix \S \ref{sec:app} provides a proof of basic properties of 
universally $H_0^S$-trivial correspondences (see Definition \ref{def:univH0Striv}).



We close this introduction with a brief discussion on related works.
Shimizu \cite{Shimizu} and Koizumi \cite{Koizumi}
obtained some results resembling our moving lemma Theorem \ref{thm:moving}
in the $\A^1$-homotopy theory.
Ayoub \cite{Ay} considered the notion of $\P^1$-localisation
which is much more sophisticated than 
our $\P^n$-Suslin complex introduced in \S \ref{sec:simplicial}.
The relation of their works with ours is to be explored.
Bruno Kahn pointed out that 
Theorem \ref{thm:moving} has implications in the theory of 
birational sheaves \cite{KS},
which should be an interesting topic for future research
(see a brief comment in Remark \ref{rem:bruno}).


\begin{ack}
The second author would like to thank Tomoyuki Abe for fruitful discussions and suggestions, which brought him into Voevodsky's theory of motives. Without that, this joint project would have never started. 
The authors thank Bruno Kahn for his intriguing comments on birational sheaves.
\end{ack}

\section{Reminders on homotopy invariant sheaves with transfers}\label{sec:HI}

We fix a field $k$.
Let $\Sch$ be the category of separated $k$-schemes of finite type,
and $\Sm$ its full subcategory of smooth $k$-schemes.
We write $\Fld_k$ for the category of fields over $k$,
and $\Fld_k^\gm$ for its full subcategory consisting of 
the $k$-fields which are isomorphic to the function field
of some (irreducible) $U \in \Sm$.
For $K \in \Fld_k$ and $X \in \Sch$, 
we write $X_K := X \otimes_k K$.

Let $\Cor$ be Voevodsky's category of finite correspondences.
By definition it has the same objects as $\Sm$,
and for $U, V \in \Sm$ 
the space of morphisms 
$\Cor(U, X)$ from $U$ to $V$ 
is the free abelian group
on the set of integral closed subschemes of $U \times X$
which is finite and surjective over an irreducible component of $U$.
An additive functor $F : \Cor^\op \to \Ab$
is called a presheaf with transfers.
Denote by $\PST$ the category of presheaves with transfers.
If $S$ is a $k$-scheme that is written as a filtered limit $S=\lim_i S_i$ 
where $S_i \in \Sm$ and all transition maps are open immersions,
then we define  
\begin{equation}\label{eq:colim-F-ext}
F(S) = \colim_i F(S_i)  \qquad (F \in \PST).
\end{equation}
We abbreviate $F(R)=F(\Spec R)$ for a $k$-algebra $R$
(when $F(\Spec R)$ is defined).
In particular, 
we may speak of $F(K)$ for $K \in \Fld_k^\gm$
and $F(\sO_{X, x})$ for $x \in X \in \Sm$.
We set $\Z_\tr(X):=\Cor(-, X) \in \PST$ for $X \in \Sm$.

We say $F \in \PST$ is homotopy invariant
if the projection $\pr : X \times \A^1 \to X$
induces an isomorphism
$\pr^* : F(X) \cong F(X \times \A^1)$ for any $X \in \Sm$.
We write $\HI$ for the full subcategory of $\PST$ of homotopy invariant
presheaves with transfers.
We say $F \in \PST$ is a Nisnevich sheaf with transfers
if $F$ composed with the inclusion (graph) functor
$\Sm \to \Cor$ is a Nisnevich sheaf on $\Sm$.
We write $\NST$ for the full subcategory of $\PST$ of 
Nisnevich sheaves with transfers.
Set $\HI_\Nis:=\HI \cap \NST$.

We will used the following facts:
\begin{enumerate}[(V1)]
\item \label{V1}
The inclusion functor $\NST \hookrightarrow \PST$
admits a left adjoint $a_\Nis : \PST \to \NST$
and $a_\Nis(\HI) \subset \HI_\Nis$ holds
\cite[Corollary 11.2, Theorem 22.3]{MVW}.
We write $F_\Nis := a_\Nis(F)$.
We have $F_\Nis(K)=F(K)$ for any $K \in \Fld_k^\gm$
(because fields are Henselian local).
\item \label{V2} 
The inclusion functor $\HI \hookrightarrow \PST$
has a left adjoint $h_0$
given by the formula 
\[
h_0(F)(U)=\Coker(F(U \times \A^1) \to F(U))
\]
for $F \in \PST, ~U \in \Sm$.
This is the maximal homotopy invariant quotient of $F$
\cite[Example 2.20]{MVW}.
For $X \in \Sm$,we write $h_0(X):=h_0(\Z_\tr(X))$.
We call
\[ H_0^S(X_K):=h_0(X)(K) = h_0(X)_\Nis(K) \]
the $0$-th \emph{Suslin homology} of $X_K$
for $K \in \Fld_k^\gm$.
There is a canonical surjective map
$H_0^S(X_K) \to \CH_0(X_K)$,
which is isomorphic if $X$ is proper over $k$
\cite[Exercise 2.21]{MVW}.
\item \label{V21} 
The inclusion functor $\HI \hookrightarrow \PST$
has a right adjoint $h^0$,
given by the formula 
\[
 h^0(F)(U)=\PST(h_0(U), F)
\]
for $F \in \PST, ~U \in \Sm$.
This is the maximal homotopy invariant subobject of $F$ \cite[\S 4.34]{RS}.
\item \label{V3} 
Let $f : F \to G$ be a morphism in $\HI_\Nis$.
If $f$ induces an isomorphism 
$f^* : F(K) \cong G(K)$ for any $K \in \Fld^\gm_k$,
then $f$ is an isomorphism in $\HI_\Nis$
\cite[Corollary 11.2]{MVW}.
\item \label{V4} 
Given $F \in \PST$, we denote by $F_\Zar$ 
the Zariski sheaf associated to the presheaf on $\Sm$
obtained by restricting $F$ along the graph functor $\Sm \to \Cor$.
In general, it does not admit a structure of presheaf with transfers,
but if $F \in \HI$
then we have $F_\Zar=F_\Nis$ by \cite[Theorem 22.2]{MVW},
and hence $F_\Zar$ acquires transfers by (V\ref{V1}).
We say $F \in \PST$ is a Zariski sheaf with transfers if $F=F_\Zar$.
\end{enumerate}

Another important fact can be stated 
as
$H^i_\Zar(-, F_\Zar)=H^i_\Nis(-, F_\Nis) \in \HI$ 
for $F \in \HI$, assuming $k$ is perfect
\cite[Theorems 5.6, 5.7]{V}.
We will not use this in the sequel.

\begin{definition}\label{def:univH0Striv}
Let $X, Y \in \Sm$.
We say $f \in \Cor(X, Y)$ is
\emph{universally $H_0^S$-trivial}
if the induced map
$f_{K*} : H_0^S(X_K) \to H_0^S(Y_K)$ 
is an isomorphism for each $K \in \Fld^\gm_k$.
\end{definition}

\begin{remark}\label{rem:univtriv-CH0-H0S}
This is an analogue of \cite[D\'efinition 1.1]{CTP},
where a proper morphism $f : X \to Y$
is said to be \emph{universally $\CH_0$-trivial}
if the induced map $f_{K*} : \CH_0(X_K) \to \CH_0(Y_K)$ 
is an isomorphism for each $K \in \Fld_k$.
When $X$ and $Y$ are proper over $k$,
a universally $\CH_0$-trivial morphism is 
also universally $H_0^S$-trivial by (V\ref{V2}).
(We tacitly identify a morphism with its graph.)
Note also that a smooth proper variety $X$ is universally $\CH_0$-trivial
(in the sense of Theorem \ref{mainthm})
if and only if the structure map $X \to \Spec k$ is universally $\CH_0$-trivial.
\end{remark}

The following result is due to 
Merkurjev \cite[Theorem 2.11]{M}
and Kahn \cite[Corollary 4.7]{K}.
We include a short proof here to keep self-containedness.

\begin{prop}\label{lem:univtriv}
Let $X, Y \in \Sm$.
The following conditions are equivalent
for $f \in \Cor(X, Y)$:
\begin{enumerate}
\item 
The finite correspondence $f$ is universally $H_0^S$-trivial.
\item
The induced map
$f_* : h_0(X)_\Nis \to h_0(Y)_\Nis$ 
is an isomorphism in $\HI_\Nis$.
\item
The induced map
$f^* : F(Y) \to F(X)$ is an isomorphism 
for each $F \in \HI_\Nis$.
\end{enumerate}
\end{prop}
\begin{proof}
The equivalence of (1) and (2)
is a consequence of (V\ref{V3}) above.
We have
\begin{align*}
F(X) 
= \PST(\Z_\tr(X), F)
= \HI(h_0(X), F)
= \HI_\Nis(h_0(X)_\Nis, F)
\end{align*}
for any $F \in \HI_\Nis$.
Here we used, in order, 
Yoneda's lemma, (V\ref{V2}), and (V\ref{V1}).
Now another use of Yoneda's lemma shows the equivalence of (2) and (3).
\end{proof}

\begin{cor}\label{cor:Mer-Kah}
Let $X$ be a smooth and proper scheme over a field $k$.
If $X$ is universally $\CH_0$-trivial,
then we have $F(k) \cong F(X)$ for any $F \in \HI_\Nis$.
\end{cor}
\begin{proof}
In view of Remark \ref{rem:univtriv-CH0-H0S},
this is a special case of Proposition \ref{lem:univtriv}.
\end{proof}

We will generalize this result in 
Corollary \ref{cor:Mer-Kah-gen} below.

\begin{remark}\label{prop:blowup}
We collect basic properties of universally $H_0^S$-trivial correspondences.
Since they are not used in the sequel,
the proof will be given in Appendix \S \ref{sec:app}.
%
\begin{enumerate}
\item
If $f, g$ are composable finite correspondences
and if two out of $f, g , f \circ g$ are universally $H_0^S$-trivial,
then so is the third.
\item
If $f : X \to Y$ and $f' : X' \to Y'$ are
universally $H_0^S$-trivial finite correspondences,
then so is $f \times f' : X \times X' \to Y \times Y'$.
\item
Suppose $k$ is perfect.
Let $j : U \hookrightarrow X$ be an open dense immersion in $\Sm$.
If $X \setminus j(U)$ is of codimension $\ge 2$, 
then $j$ is universally $H_0^S$-trivial.
\item 
Suppose $k$ is perfect.
A proper birational morphism in $\Sm$
is universally $H_0^S$-trivial.
\end{enumerate}
\end{remark}

\section{$\P^1$-invariance and the main result}\label{sec:PI}

Recall from \cite[\S 3.2]{voetri}
that $\PST$ is equipped with a
symmetric monoidal structure $\otimes$
which is uniquely characterized by the facts 
that it is right exact and that the Yoneda functor is monoidal
(that is, $\Z_\tr(X) \otimes \Z_\tr(Y)=\Z_\tr(X \times Y)$
for $X, Y \in \Sm$).
It admits a right adjoint $\uHom$ given by the formula
\begin{equation}\label{eq:def-uHom}
\uHom(F, G)(X) = \PST(F \otimes \Z_\tr(X), G)
\qquad (F, G \in \PST, ~X \in \Sm).
\end{equation}

\begin{definition}\label{def:PI}
We say $G \in \PST$ is \emph{$\P^1$-invariant}
if the structure map $\sigma : \P^1 \to \Spec k$
induces an isomorphism 
$G \overset{\cong}{\longrightarrow} \uHom(\Z_\tr(\P^1), G)$,
that is,
$\sigma$ induces isomorphisms 
$G(U) \overset{\cong}{\longrightarrow} G(U \times \P^1)$
for all $U \in \Sm$.
Denote by $\PI_\Nis$ the full subcategory of $\PST$
consisting of all Nisnevich sheaves with transfers
which are $\P^1$-invariant.
\end{definition}

\begin{thm}\label{prop:reduction-moving}
Suppose that $X, Y$ are 
smooth and proper schemes over a field $k$
and $f \in \Cor(X, Y)$.
Then the conditions in Proposition \ref{lem:univtriv}
are equivalent to the following:
\begin{enumerate}
\item[(4)]
The induced map
$f^* : G(Y) \to G(X)$ is an isomorphism 
for each $G \in \PI_\Nis$.
\end{enumerate}
\end{thm}

As with Corollary \ref{cor:Mer-Kah},
Theorem \ref{prop:reduction-moving} has an immediate consequence:

\begin{cor}\label{cor:Mer-Kah-gen}
Let $X$ be a smooth proper scheme over a field $k$.
If $X$ is universally $\CH_0$-trivial,
then we have $G(k) \cong G(X)$
for any $G \in \PI_\Nis$.
\end{cor}

\begin{remark}\label{rem:compare}
Let $T$ be a smooth quasi-affine scheme over $k$.
It follows from \cite[Theorem 6.4.1]{kmsy3} that
$\Z_\tr(T) \in \PI_\Nis$.
This shows that 
$\P^1$-invariance does not imply reciprocity
(i.e. the converse of the second arrow \eqref{eq:HI-Rec-PI} does not hold).
On the other hand, the conclusion of Corollary \ref{cor:Mer-Kah-gen}
is obvious for $G=\Z_\tr(T)$.
Indeed,
it is not difficult to show
$\Cor(\Spec k, T) \cong \Cor(X, T)$
for any $X \in \Sm$ which is 
connected and proper over $k$
(but not necessary universally $\CH_0$-trivial).
\end{remark}

For $F \in \PST$ we define
\begin{equation} \label{eq:def-olh0}
\ol{h}_0(F) := 
\Coker(i_0^* - i_1^* : \uHom(\Z_\tr(\P^1), F) \to F).
\end{equation}
We write $\ol{h}_0(X):=\ol{h}_0(\Z_\tr(X))$ for $X \in \Sm$.
The main part in the proof of
Theorem \ref{prop:reduction-moving} is the following:

\begin{theorem}[A moving lemma]\label{thm:moving}
We have $\ol{h}_0(X)_\Zar \cong h_0(X)_\Zar$
for any smooth proper scheme $X$ over a field $k$.
(Hence we have $\ol{h}_0(X)_\Nis \cong h_0(X)_\Nis$ as well.)
\end{theorem}

The proof of Theorem \ref{thm:moving}
occupies the next two sections.
In the rest of this section,
we deduce Theorem \ref{prop:reduction-moving}
assuming Theorem \ref{thm:moving}.

\begin{remark}\label{rem:bruno}
For $X$ as in Theorem \ref{thm:moving},
we have an explicit formula:
\[  \ol{h}_0(X)_\Zar(U) \cong \CH_0(X \otimes_k k(U))
\quad \text{for any connected}~U \in \Sm,
\]
because we know
$h_0(X)_\Nis(U) \cong \CH_0(X \otimes_k k(U))$
by \cite[Theorem 3.1.2]{KS}
(and 
we have $h_0(X)_\Zar \cong h_0(X)_\Nis$
by (V\ref{V2}) and (V\ref{V4})).
In particular, 
$\ol{h}_0(X)_\Zar$ is \emph{birational} in the sense of
\cite[Definition 2.3.1]{KS},
that is, 
any open dense immersion $V \hookrightarrow W$ induces an isomorphism 
$\ol{h}_0(X)_\Zar(W) \cong \ol{h}_0(X)_\Zar(V)$.
\end{remark}


Denote by $i_\epsilon : \Spec k \to \P^1$
the closed immersion defined by a rational point
$\epsilon \in \P^1(k)$.

\begin{definition}
Let $F \in \PST$.
We say $F$ is \emph{$\P^1$-rigid} if 
the two induced maps 
\[ i_0^*, ~i_1^* : F(U \times \P^1) \to F(U) \] 
are equal for any $U \in \Sm$.
Denote by $\PR$ the full subcategory of $\PST$
consisting of all $\P^1$-rigid presheaves with transfers.
\end{definition}

\begin{lemma}\label{lem:PI-PRig}
If $F \in \PST$ is $\P^1$-invariant, then it is $\P^1$-rigid.
The converse holds if $F$ is separated for Zariski topology.
\end{lemma}
\begin{proof}
See \cite[Proposition 6.1.4]{rec}.
\end{proof}

\begin{lemma}
\begin{enumerate}
\item 
For $F \in \PST$, 
the following conditions are equivalent:
\begin{enumerate}
\item 
$F$ is $\P^1$-rigid.
\item 
The two induced maps 
$i_0^*, ~i_1^* : \uHom(\Z_\tr(\P^1), F) \to F$ are equal.
\item
The canonical surjection $F \to \ol{h}_0(F)$
is an isomorphism.
\item
The canonical injection
$\PST(\ol{h}_0(U), F) \to \PST(\Z_\tr(U), F)$ 
is an isomorphism for each $U \in \Sm$.
\end{enumerate}
\item
The formula \eqref{eq:def-olh0}
defines
a left adjoint $\ol{h}_0 : \PST \to \PR$
of the inclusion functor $\PR \hookrightarrow \PST$.
\item
We have $\HI_\Nis \subset \PI_\Nis$.
\end{enumerate}
\label{lem:P1rigid}
\end{lemma}

\begin{proof}
(1) The equivalence of (a), (b) and (c)
follows from \eqref{eq:def-uHom} and \eqref{eq:def-olh0}.
If (c) holds, then any morphism
$\Z_\tr(U) \to F$ factors as
$\Z_\tr(U) \twoheadrightarrow \ol{h}_0(U) \to \ol{h}_0(F) \cong F$,
whence (d).
If (d) holds, then any morphism
$\Z_\tr(U) \to F$ factors as
$\Z_\tr(U) \twoheadrightarrow \ol{h}_0(U) \to F$,
whence (b).

(2) We need to show 
$\PST(\ol{h}_0(G), F) \cong \PST(G, F)$
for any $F \in \PR$ and $G \in \PST$.
This is (d) above if $G=\Z_\tr(U)$ for $U \in \Sm$,
to which the general case is reduced
by taking a resolution of the form
\[ \bigoplus_\beta \Z_\tr(V_\beta) \to \bigoplus_\alpha \Z_\tr(U_\alpha)
\to G \to 0,
\]
where $U_\alpha, ~V_\beta \in \Sm$
(see \cite[\S 3.2]{voetri}).

(3) 
Given $F \in \PST$,
we have a chain of canonical surjections
$F \twoheadrightarrow \ol{h}_0(F) \twoheadrightarrow h_0(F)$,
and $F$ is homotopy invariant if and only if the composition is an isomorphism.
It follows from (1) that $\HI \subset \PRig$.
Now the assertion follows from Lemma \ref{lem:PI-PRig}.
\end{proof}

\begin{lemma}\label{lem:sheaf}
Suppose $G \in \PST$ is a Zariski sheaf.
Then $h^0(G)$ is a Nisnevich sheaf.
\end{lemma}
\begin{proof}
This is essentially shown in \cite[\S 4.34]{RS}, 
but we include a short proof for the completeness sake.
We consider a commutative diagram:
\[
\xymatrix{
h^0(G) \ar@{^{(}->}[r] \ar[d]_j & 
G \ar@{=}[d]
\\
h^0(G)_\Zar \ar@{^{(}->}[r] & 
G_\Zar
}
\]
The bottom arrow is injective since the sheafification is exact.
This shows the injectivity of $j$.
The fact (V\ref{V4}) shows  $h^0(G)_\Zar=h^0(G)_\Nis$, 
and (V\ref{V1}) shows it is homotopy invariant.
Hence $j$ must be isomorphic since $h^0(G) \subset G$
is the maximal subobject in $\HI$.
\end{proof}

\begin{proof}
[Proof of Theorem \ref{prop:reduction-moving}, admitting Theorem \ref{thm:moving}]
That (4) implies (3) follows from Lemma \ref{lem:P1rigid} (3).
To show the converse, 
we assume (3) and take $G \in \PI_\Nis$. 
Set $F:=h^0(G)$.
By Lemma \ref{lem:sheaf} we find $F \in \HI_\Nis$,
and hence we have  $f^* : F(Y) \cong F(X)$ by the assumption (3).
It remains to show  $G(X)=F(X)$ for any proper $X \in \Sm$.
By Theorem \ref{thm:moving} we have
$\ol{h}_0(X)_\Nis=h_0(X)_\Nis \in \HI_\Nis \subset \PST$. 
We now proceed as follows:
\begin{align*}
G(X) 
&=\PST(\Z_\tr(X), G)& 
&\text{Yoneda}&
\\
&=\PST(\ol{h}_0(X), G)&
&\text{Lemma \ref{lem:P1rigid} (d)}&
\\
&=\PST(\ol{h}_0(X)_\Nis, G)&
&\text{$G$ is a Nisnevich sheaf}&
\\
&=\PST(h_0(X)_\Nis, G)&
&\ol{h}_0(X)_\Nis=h_0(X)_\Nis&
\\
&=\PST(h_0(X), G)&
&\text{$G$ is a Nisnevich sheaf}&
\\
&=\HI(h_0(X), F) &
&\text{$h^0$ is a right adjoint to the inclusion}&
\\
&=\PST(\Z_\tr(X), F) &
&\text{$h_0$ is a left adjoint to the inclusion}&
\\
&=F(X)&
&\text{Yoneda}.&
\end{align*}
This completes the proof.
\end{proof}

\section{Projective Suslin complex}\label{sec:simplicial}

Fix $X \in \Sm$ in this section.
After a brief review of the definition
of the Suslin complex of $X$,
we define its variant using the projective spaces $\P^n$.
This will be used in the proof of the moving lemma in the next section.
(We will use them only up to degree two).

For each non-negative integer $n$, we write
\[ \Delta^n :=\Spec k[t_0, \dots, t_n]/(t_0+\cdots+t_n-1) \cong \A^n.
\]
For $j=0, \dots, n$, we define
\begin{equation}\label{def:facemaps}
i_{n, j} : \Delta^{n-1} \to \Delta^n;
\quad
(t_0, \dots, t_{n-1}) \mapsto (t_0, \dots, t_{j-1}, 0, t_j, \dots, t_{n-1}).
\end{equation}
The \emph{Suslin complex} $C_\bullet(X)$ of $X$
is a complex in $\PST$ defined by
\begin{align*}
&C_n(X):=\uHom(\Z_\tr(\Delta^n), \Z_\tr(X)),
\\
& \pd_n := \sum_{j=0}^n (-1)^j i_{n, j}^* 
: C_n(X) \to C_{n-1}(X).
\end{align*}
Its homology is denoted by $h_n(X) \in \PST$. 
For $n=0$, it recovers $h_0(X)$ from (V\ref{V2}).

For each non-negative integer $n$, we put
\[ 
\ol{\Delta}^n:= \Proj k[t_0, \dots, t_{n+1}]/(t_0+\cdots+t_n-t_{n+1})
\cong \P^n.
\]
We have a canonical open immersion
$\iota_n : \Delta^n \hookrightarrow \ol{\Delta}^n$,
which is isomorphic to $\A^n \hookrightarrow \P^n$.
It induces a map in $\PST$
\begin{equation}\label{eq:incl-olC-C}
\ol{C}_n(X):=\uHom(\Z_\tr(\ol{\Delta}^n), \Z_\tr(X))
\hookrightarrow C_n(X),
\end{equation}
which is injective for any $n$ and isomorphic for $n=0$.
Indeed, its section over $U \in \Sm$ is given by 
\[ (\id_U \times \iota_n)^* 
: \Cor(U \times \ol{\Delta}^n, X) \to \Cor(U \times \Delta^n, X),
\]
which injective in general and isomorphic for $n=0$.
We regard $\ol{C}_n(X)$ as a subobject in $\PST$ of $C_n(X)$.
Since the morphisms $i_{n, j}$ from \eqref{def:facemaps}
extends (uniquely) to morphisms $\ol{\Delta}^{n-1} \to \ol{\Delta}^{n}$,
we obtain a subcomplex $\ol{C}_\bullet(X)$ of $C_\bullet(X)$.
We write its homology by $\ol{h}_n(X) \in \PST$.
For $n=0$, it recovers $\ol{h}_0(X)$ from \eqref{eq:def-olh0}.

We write
\begin{equation}\label{eq:def-Qcpx}
Q_\bullet(X):=C_\bullet(X)/\ol{C}_\bullet(X) 
\end{equation}
for the quotient complex,
and its homology presheaf
is denoted by $H_n(Q_\bullet(X)) \in \PST$.
We have $H_0(Q_\bullet(X))=0$,
for \eqref{eq:incl-olC-C} is isomorphic for $n=0$.
Since the sheafification is exact,
we obtain  an exact sequence
\[ H_1(Q_\bullet(X))_\Zar
\to \ol{h}_0(X)_\Zar \to h_0(X)_\Zar \to H_0(Q_\bullet(X))_\Zar=0.
\]
Theorem \ref{thm:moving} is now reduced to the following.

\begin{theorem}\label{thm:moving2}
If $X \in \Sm$ is proper, then we have $H_1(Q_\bullet(X))_\Zar=0$.
\end{theorem}

\begin{remark}
\begin{enumerate}
\item 
In general,
$\ol{h}_n(X)_\Zar \to h_n(X)_\Zar$ is not isomorphic for $n>0$.
Indeed, one easily checks 
$H_2(Q_\bullet(\P^1))_\Zar(k)=H_2(Q_\bullet(\P^1))(k) \not= 0$.
\item 
The properness assumption on $X$ is essential.
Indeed, it follows from
\cite[Theorem 6.4.1]{kmsy3}
that if $X$ is quasi-affine
then all the boundary maps of $\ol{C}_\bullet(X)$ are the zero maps.
\end{enumerate}
\end{remark}

\section{Moving lemma}\label{sec:moving lemma}

%

We shall prove Theorem \ref{thm:moving2} 
in the following (equivalent) form:

\begin{theorem}\label{thm:moving-lemma}
Let $X\in \Sm $ be proper.
   For every irreducible affine $V\in \Sm $ and local scheme $U$ at a closed point of $V$, the restriction map
\[ H_1(\Qbullet )(V) \to H_1(\Qbullet )(U)\]
is the zero map.
(See \eqref{eq:colim-F-ext} for the 
definition of $H_1(\Qbullet )(U)$.)
\end{theorem}

Note that in proving Theorem \ref{thm:moving-lemma}, we may assume $k$ is infinite;
for if $k$ is finite we can use the usual norm argument.

\subsection{The bad locus}

In the notation of Theorem \ref{thm:moving-lemma},
let $\Gamma \subset V\times X\times \Delta ^n$ be an irreducible closed subset which is finite and surjective over $V\times \Delta ^n$.
Let $\ol \Gamma $ be its closure in $V\times X\times \olDelta n$.
We call
\begin{equation}\label{eq:def-of-BGamma}
    B(\Gamma ):=\br{p\in V\times \olDelta n \midd \olGamma \to V\times \olDelta n \text{ is not finite over $p$}}
\end{equation}
the \emph{bad locus} of $\Gamma$,
which witnesses how far $\Gamma $ is from being a member of $\ol C _n(X)(V)$.

\begin{lemma}
    \begin{enumerate}[(i)]
        \item\label{item:witness} We have $\Gamma \in C_n(X)(V)$ if and only if $B(\Gamma )=\emptyset $.
        \item\label{item:proper-subset} 
The bad locus $B(\Gamma )$ is a closed proper subset of $V\times (\olDelta n\setminus \Delta ^n)$.
        \item\label{item:still-proper-subset} If $n\le 1$, the image of the projection $B(\Gamma )\to V$ is a closed proper subset.
    \end{enumerate}
    \label{lem:bad-locus}
\end{lemma}
\begin{proof}
    The assertion \eqref{item:witness} is clear from definitions and put for later reference.

    To prove \eqref{item:proper-subset},
    consider the set upstairs:
    \[ \widetilde B(\olGamma ):= \br{ x\in \olGamma  \midd
    \text{locally around $x$, the fiber of $\olGamma \to V\times \olDelta n$ has dimension $\ge 1 $}
    }.
    \]
    It is a closed subset of $V\times X\times \olDelta n$ by Chevalley's theorem \cite[IV${}_{3}$ 13.1.3]{EGA}.
    It is contained in $\olGamma \setminus \Gamma $:
    \begin{equation}\label{eq:contained-in-difference}
        \widetilde{B}(\olGamma ) \subset \olGamma \setminus \Gamma ,
    \end{equation}
    because $\Gamma $ is assumed to be finite over $V\times \Delta ^n$.
    Since $B(\Gamma )$ is by definition the image of the map $\widetilde B(\olGamma )\to V\times \olDelta n$ which is proper because $X$ is, it follows that $B(\Gamma )$ is a closed subset of $V\times (\olDelta n\setminus \Delta ^n)$.

    To show $B(\Gamma )$ is a proper subset of $V\times (\olDelta n\setminus \Delta ^n)$, let $\xi \in \widetilde B(\olGamma )$ be an arbitrary point and $\eta \in B(\Gamma )$ its image.
    By \eqref{eq:contained-in-difference} we have
    \begin{align*} \dim (\olGamma \setminus \Gamma )&\ge \trdeg (k(\xi )/k)
        \\
        &=\trdeg (k(\xi )/k(\eta ) ) +\trdeg (k(\eta )/k)
        \\
        &\ge 1+\trdeg (k(\eta )/k ),
    \end{align*}
and by $\dim (V)+n-1 \ge \dim (\olGamma \setminus \Gamma ) $ we obtain
\begin{equation}\label{eq:inequality-obtained}
    \dim (V)+(n-2) \ge \trdeg (k(\eta )/k) .
\end{equation}
Since $\dim (B(\Gamma ))= \sup _{\eta \in B(\Gamma )} \trdeg (k(\eta )/k) $
we conclude $\dim (V\times (\olDelta n \setminus \Delta ^n))>\dim (B(\Gamma ))$.
We are done.

The assertion \eqref{item:still-proper-subset} is a direct consequence of \eqref{item:proper-subset} (or of \eqref{eq:inequality-obtained}).
\end{proof}

\subsection{Affine space case}\label{sec:affine-space-case}

In this subsection we consider the case $V=\A ^N$ of Theorem \ref{thm:moving-lemma}, with $N\ge 0$ an integer.
Recall we may assume $k$ is infinite, which we do here.

Let $\Gamma \in C_1(X)(\A ^N)$ be an arbitrary irreducible cycle.
By a diagram chase in the diagram below 
(see \eqref{eq:def-Qcpx} for the definition of $Q_\bullet(X)$),
it suffices to find $\tildegamma \in C_2(X)(\A ^N)$ such that
$(\Gamma -\partial_2 \widetilde\Gamma )|_U \in \ol C_1(X)(U)$: 
\[ \xymatrix{
    \tildegamma \ ? \ar@{}|{\downin}[d]& \Gamma \ar@{}|{\downin}[d]
    \\
    Q_2(X)(\A^N) \ar[r]^{\partial_2}\ar[d] & Q_1(X)(\A ^N)\ar[r]\ar[d] & 0
    \\
    Q_2(X)(U) \ar[r]^{\partial_2} & Q_1(X)(U)\ar[r] & 0 .
} \]

For a vector $\bm v\in \A^N(k)$, consider the translation $+\bm v : \A^N \to \A^N$ by $\bm v$.
The next assertion suggests how it can be useful.

\begin{lemma}\label{lem:bad-vectors}
    There is a closed proper subset $B\subset \A^N$ such that for every vector $\bm v\in \A^N(k)\setminus B$,
    if we denote by $\tau _{\bm v}\colon \A ^N\times X\times \A^1\to \A^N\times X\times \A^1$
    the base change of the translation $\A^N \xrightarrow{+\bm v} \A^N$, then we have
    \[ (\tau _{\bm v}^*\Gamma )|_U \in \ol C_1(X)(U) .\]
\end{lemma}
\begin{proof}
    Let $s\in U$ be the unique closed point of $V$ contained in $U$.
    Let $B(\Gamma )\subset V\times \olDelta 1$ be as in \eqref{eq:def-of-BGamma}.
    By Lemma \ref{lem:bad-locus} we know that its projection $\pr_V (B(\Gamma ))\subset V=\A ^N $ is a closed proper subset.
    Consider the closed subset
    \[ B_0:=\pr_V  (B(\Gamma )_{k(s)})-s = \br{\bm v\in \A^N_{k(s)} \midd \bm v+s \in \pr_V  B(\Gamma ) _{k(s)} } \subsetneq  \A^N_{k(s)}\]
    and let $B$ be its image by the finite projection $\pi \colon \A^N_{k(s)}\to\A^N$:
    \begin{equation}
        B:= \pi  (B_0) \subsetneq \A^N.
    \end{equation}
Take an arbitrary $\bm v\in \A^N(k) \setminus B$.
    By definitions, we know $s+ \bm v \in \A^N\setminus \pr_V B(\Gamma )$. Since the right hand side is an open subset of $\A^N$, this relation remains true if we replace $s$ by any point specializing to $s$. In particular:
    \begin{equation}\label{eq:translate-not-in-pB}
        U+\bm v\subset \A ^N\setminus \pr_V B(\Gamma ).
    \end{equation}

    Let us denote by
    $\ol\tau_{\bm v}$
    the endomorphism of $\A^N \times X\times \olDelta 1$ obtained as the base change of the translation $+\bm v$.
    By \eqref{eq:translate-not-in-pB} the maps $\tau _{\bm v}$ and $\ol\tau _{\bm v}$ restrict themselves as in the following commutative diagram:
    \[ \xymatrix{
        U\times X\times \Delta ^1 \ar[r]^(0.4){\tau _{\bm v}}\ar@{}|{\cap}[d]
        &(\A^N\setminus \pr_V B(\Gamma ))\times X\times \Delta ^1
        \ar@{}|{\cap}[d]
        &\Gamma |_{\A^N\setminus \pr_V B(\Gamma )}
        \ar@{}|(0.4){\supset }[l]
        \ar@{}|{\cap}[d]
        \\
        U\times X\times \olDelta 1
        \ar[r]^(0.4){\ol\tau _{\bm v}}
        \ar[d]^{\pr }
        &(\A^N\setminus \pr_V B(\Gamma ))\times X\times \olDelta 1
        \ar[d]_{\pr}
        &\olGamma |_{\A^N\setminus \pr_V B(\Gamma )} \ar@{}|(0.4){\supset }[l]\ar[dl]^{\text{finite}}
        \\
        U\times \olDelta 1 \ar[r]^(0.35){(+\bm v) \times \id }
        &(\A^N\setminus \pr_V B(\Gamma ))\times \olDelta 1.
        &
    }\]
    Here, the slanted arrow in the diagram is finite
    because $\olGamma \to \A^N\times \olDelta 1$ is finite outside $B(\Gamma )$
    precisely by the definition \eqref{eq:def-of-BGamma}
    and we have the inclusion $(\A^N\setminus \pr_V B(\Gamma ))\times \olDelta 1 \subset (\A^N\times \olDelta 1)\setminus B(\Gamma )$.
    It follows that $\ol\tau _{\bm v}^* \olGamma $ is finite over $U\times  \olDelta 1$.
    As a general fact about closure and continuity, we have the inclusion $\ol{(\tau _{\bm v}^*\Gamma )}\subset \ol\tau _{\bm v}^* \olGamma $.
    We conlude that $(\tau _{\bm v}^*\Gamma )|_{U}$ belongs to $\ol C_1(X)(U)$ and this completes the proof.
\end{proof}

Let $B$ be as in Lemma \ref{lem:bad-locus} and fix a vector $\bm v\in \A^N(k)\setminus B$.
Let $\varphi _{\bm v}\colon \A^N \times \A^1\to \A^N$ be the map $(a,t)\mapsto a+t\bm v$ and let
\[ \Phi _{\bm v,n}\colon \A^N\times X\times \A^1\times \Delta ^n\to \A^N\times X\times  \Delta ^n \]
be its base change by $X\times \Delta ^n \to \Spec k$.
Since finite morphisms are stable under base change, we know:
\begin{equation}\label{eq:inverse-image-is-finite}
    \text{the inverse image }\Phi _{\bm v,1}^{-1}\Gamma \subset \A^N\times X\times \A^1\times \Delta ^1 \text{ is finite over }\A^N\times \A^1\times \Delta ^1
.\end{equation}
%
%
%

\input{ClosureMovingTikZ2}
We shall use the following triangulation maps as in Figure \ref{fig:triangulation}:
\begin{equation}
    \sigma _1,\sigma _2 \colon \Delta ^2 \xrightarrow{\cong } \A^1 \times \Delta ^1.
\end{equation}
Explicitly, the maps $\sigma_1, \sigma_2$ are the unique affine-linear ones satisfying 
\begin{align*}
&\sigma _1(u_0)=(v_0,0),& 
&\sigma _1(u_1)=(v_1,0),& 
&\sigma _1(u_2)=(v_1,1);& 
\\
&\sigma _2(u_0)=(v_0,0),& 
&\sigma _2(u_1)=(v_0,1),& 
&\sigma _2(u_2)=(v_1,1)& 
\end{align*}
in $(\Delta^1\times \A^1)(k)$,
where $v_j=i_{1, j}(\Delta^0)$ for $j=0, 1$
and
$u_0=i_{2, 2}(v_1), ~
u_1=i_{2, 0}(v_1), ~
u_2=i_{2, 1}(v_0)$.


Let $\Phi ^{(1)}_{\bm v,1}, \Phi ^{(2)}_{\bm v,1}\colon \A^N\times X\times \Delta ^2 \to \A^N\times X\times \Delta ^1$ be the composite maps
$\Phi _{\bm v,1}\circ (\id _{\A^N\times X}\times \sigma _i)$ ($i=1,2$).
By \eqref{eq:inverse-image-is-finite} we conclude
$\Phi ^{(1)*}_{\bm v,1} \Gamma ,\ \Phi ^{(2)*}_{\bm v,1} \Gamma  \in C_2(X)(\A^N) .$ 
Now set 
\[ \widetilde{\Gamma }:= \Phi ^{(1)*}_{\bm v,1} \Gamma -\ \Phi ^{(2)*}_{\bm v,1} \Gamma  .\] 
We want to show $(\Gamma -\partial_2 \widetilde{\Gamma })|_U \in \ol C_1(X)(U)$.

By a routine calculation of $\partial_2 \widetilde{\Gamma }$ we have
\begin{equation}\label{eq:derivative-of-Gamma-tilde}
    \Gamma - \partial_2 \widetilde\Gamma = \tau _{\bm v}^*\Gamma
    +\Phi _{\bm v,0}^*(i_{1,1}^*\Gamma )
    -\Phi _{\bm v,0}^*(i_{1,0}^*\Gamma )
.\end{equation}
By Lemma \ref{lem:bad-vectors}, for $\bm v\in \A^N(k)\setminus B $ we know that the first term of the right hand side maps into $\ol C_1(X)(U)$.
So it suffices to show the following (which we apply to $\gamma := i_{1,j}^*\Gamma $):
\begin{lemma}\label{lem:bad-locus-lower-dim}
    For every $\gamma \in C_0(X)(\A^N ) $, there is a closed proper subset $C\subsetneq \A^N$ such that for all $\bm v\in \A^N(k)\setminus C$
    the cycle $\Phi _{\bm v,0}^*\gamma $ on $\A^N\times X\times \Delta ^1$ belongs to $\ol C_1(X)(\A^N)$.
\end{lemma}
\begin{proof}
    Let us observe that as long as $\bm v\neq 0$ the morphism $\varphi _{\bm v}$ extends (uniquely) to a morphism
    \begin{align}\label{eq:phi-written-as}
        \ol\varphi _{\bm v}\colon \A^N\times \bP ^1 &\to \bP^N \\
\notag
        (a,[t_0:t_1])&\mapsto [t_0:t_0a+t_1\bm v].
    \end{align}

%
%

    For the proof of Lemma \ref{lem:bad-locus-lower-dim} we may assume $\gamma $ is irreducible.
    Let $\ol\gamma \subset \bP^N\times X$ be the closure of $\gamma $.
    Consider the set:
    \begin{equation}\label{eq:def-of-C-infty}
        C_\infty := \br{ a\in \bP^N\setminus \A^N \midd \ol\gamma \to \bP^N
        \text{ is not finite over }a } .
    \end{equation}
    By the $V=\Spec (k)$ case of Lemma \ref{lem:bad-locus} \eqref{item:proper-subset} (via $X\times \olDelta n \cong \bP ^N\times X$),
    we find that $C_\infty $ is a closed proper subset of $\bP^N\setminus\A^N$.
    Let $C\subset \A^N$ be the cone associated to $C_\infty $, namely,
    \[ C:= 
\left\{
\begin{array}{ll}
\{ 0 \} \cup q^{-1}(C_\infty) \subset \A^N 
& (N>0), 
\\
\emptyset & (N=0),
\end{array}
\right.
\]
where $q : \A^N\setminus\br 0 \to \bP^N\setminus\A^N$ is 
the projection with center $0$.
This is a closed proper subset of $\A^N$.
%

    Now suppose $\bm v\in \A^N(k)\setminus C$. Then by \eqref{eq:phi-written-as}
    we see $\ol\varphi _{\bm v}$ maps $\A^N \times \P^1$ into $\bP^N\setminus C_\infty $.
    We obtain the following commutative diagram:
    \[
        \xymatrix{
        \A^N \times X\times \Delta ^1
        \ar[r]^{\Phi _{\bm v,0}}
        \ar@{}|{\cap}[d]
        &
        \A^N\times X
        \ar@{}|{\cap}[d]
        &
        \gamma \ar@{}|{\supset}[l]
        \ar@{}|{\cap}[d]
        \\
        \A^N \times X\times \bP ^1
        \ar[r]^{\ol\varphi _{\bm v}\times \id _X}
        \ar[d]^{\pr }
        &
        (\bP ^N\setminus C_\infty )\times X
        \ar[d]^{\pr }
        &
        \ol\gamma |_{\bP ^N\setminus C_\infty }
        \ar@{}|{\supset}[l]
        \ar[dl]
        \\
        \A^N\times \bP^1
        \ar[r]^{\ol\varphi _{\bm v}}
        &
        \bP ^N\setminus C_\infty .
        }
    \]
    Since the slanted arrow is finite by the definition \eqref{eq:def-of-C-infty} of $C_\infty $,
    the inverse image $(\ol\varphi _{\bm v}\times \id _X)^*\ol\gamma $ is finite over $\A^N \times \P^1$.
    By the inclusion $\ol{(\Phi _{\bm v,0}^*\gamma )}\subset (\ol\varphi _{\bm v}\times \id _X)^*\ol\gamma $ of subsets of $\A^N\times X\times \bP^1$,
    we conclude $\Phi _{\bm v,0}^*\gamma $ belongs to $\ol C_1(X)(\A^N)$,
    completing the proof.
\end{proof}
Lemmas \ref{lem:bad-vectors} and \ref{lem:bad-locus-lower-dim} applied to \eqref{eq:derivative-of-Gamma-tilde} prove Theorem \ref{thm:moving-lemma} in the case $V=\A^N$.

\subsection{The general case}

Let $V\in \Sm $ be affine and irreducible.
Let $N=\dim (V)$ be its dimension.
Fix a closed embedding $V\hookrightarrow \A^{N'}$ into an affine space.
For a technical reason (cf.\ Proposition \ref{prop:Chow}) we assume it is obtained as the composition of a preliminary one $V\hookrightarrow \A ^{N'_0}$
and the $2$-fold Veronese embedding $\A ^{N'_0}\hookrightarrow \A ^{N'}$, 
$N':= \begin{pmatrix}
    N'_0+2 \\ 2
\end{pmatrix} 
-1 $
defined by
\begin{equation}\label{eq:Veronese}
    (x_i)_{i=1,\dots ,N'_0}
    \mapsto
    (x_1,\dots ,x_{N'_0} ; (x_ix_j)_{i\le j} ).
\end{equation}
Let $M_{NN'}\cong \A^{NN'}$ be the $k$-scheme parametrizing $N\times N'$ matrices.
The choice of a $k$-rational point $f\in M_{NN'}(k)$ determines a morphism which we denote by the same symbol $f\colon \A^{N'}\to \A^N$.

\begin{proposition}[Noether's normalization lemma]\label{lem:Noether}
    There is a closed proper subset $D_1\subset M_{NN'}$ such that the composite map
    \[ \pi _f \colon V\hookrightarrow \A^{N'}\xrightarrow[]{f}\A^N \]
    is finite and flat whenever $f\in M_{NN'}(k)\setminus D_1$.
\end{proposition}
\begin{proof}
    For the existence of $D_1$ which guarantees finiteness, see e.g.\ \cite[p.\ 69]{AM} or \cite[\S 3.1]{Kai}.
    Flatness is then automatic by the smoothness of $V$ and $\A^N$;
    see \cite[Exercise III-10.9, p.276]{Hartshorne} or \cite[IV${}_2$ 6.1.5]{EGA}.
\end{proof}

By Proposition \ref{lem:Noether}, for $f\in M_{NN'}(k)\setminus D_1$ we have pushforward maps $\pi _{f*}\colon H_n(X)(V)\to H_n(X)(\A^N)$.
Let $s\in U$ be the unique closed point of $V$ contained in $U$.
Let $U_0\subset \A^N$ be the local scheme at $\pi _f (s)$.
Since $\pi _f$ carries $U$ into $U_0$ we have the following commutative diagram:
\begin{equation}
    \xymatrix{
    C_1(X)(V)
    \ar@{->>}[r]
    &H_1(Q_\bullet(X))(V)
    \ar[d]^{\pi _{f*}}
    \\
    &H_1(Q_\bullet(X))(\A^N)
    \ar[d]^{\pi _{f}^*}
    \ar[r]^{(-)|_{U_0} =0}
    &H_1(Q_\bullet(X))(U_0)
    \ar[d]^{(\pi _f|_{U})^*}
    \\
    &H_1(Q_\bullet(X))(V )
    \ar[r]^{(-)|_{U}}
    &H_1(Q_\bullet(X))(U)
    .}
\end{equation}
Here the restriction map $(-)|_{U_0}$ is the zero map by the conclusion of Section \ref{sec:affine-space-case}.

Now toward the proof of Theorem \ref{thm:moving-lemma},
let 
\[ \Gamma \in C_1(X)(V) \]
be an irreducible cycle. The commutative diagram shows that 
$ (\pi _f^*\pi _{f*} \Gamma )|_U =0$ in $H_1(Q_\bullet(X))(U)$.
In other words, 
we have
\begin{equation}\label{eq:tautologically}
    \Gamma |_U = (\Gamma - \pi _f^*\pi _{f*} \Gamma )|_U \quad\text{ in }H_1(Q_\bullet(X))(U).
\end{equation} 
This right hand side turns out to be easier to handle.

Let $\ol\Gamma \subset V\times X\times \ol\Delta ^1$ be the closure and
le $B(\Gamma )\subset V\times \olDelta 1$ be the bad locus in \eqref{eq:def-of-BGamma}.
By Lemma \ref{lem:bad-locus} we know it is a closed proper subset of $V\times \br{\infty }$.
Let $\ol B (\Gamma )\subsetneq V$ be its projection (which is isomorphic to $B(\Gamma )$).

\begin{proposition}\label{prop:Chow}
    There is a closed proper subset $D_2\subsetneq M_{NN'}$ such that whenever
    $f\in M_{NN'}(k)\setminus (D_2\cup D_1)$ we have the equality of zero cycles:
    \begin{equation}\label{eq:Chow}
        \pi _f^*\pi _{f*}s = s+\sum _{i=1}^m x_i \quad\text{ on }V, 
    \end{equation} 
        where $s,x_1,\dots ,x_m$ are distinct and $x_i\in V\setminus \ol B(\Gamma )$ for all $i$.
\end{proposition}
\begin{proof}
    This can be shown using Chow's techniques \cite[pp.\ 458--460]{Chow}, \cite[Lamma 2 on p.\ 3-08]{Chevalley}.
    We present a proof based on a more recent account \cite{Kai}.

    First, since we are using the Veronese embedding \eqref{eq:Veronese} we can invoke
    \cite[Propositions 3.2 and 3.3]{Kai} which state that there is a closed proper subset $D'\subsetneq M_{NN'}$ such that for all $f\in M_{NN'}(k)\setminus (D'\cup D_1)$ the map $\pi _f\colon V\to \A^N $ is \'etale over $\pi _f(s)$ and the restriction $s\to \pi _f(s)$ is an isomorphism.
    This gives an equality of the form \eqref{eq:Chow} with $s$ and $x_i$'s distinct.

    It remains to show $x_i\in V\setminus \ol B(\Gamma )$.
    We need the following statement.
    \begin{proposition}[{a special case of \cite[Proposition 3.5]{Kai}}]\label{prop:Chow-in-Kai}
        Let $B\subset V$ be a proper closed subset and 
        $s\in V$ a closed point.
        Then there is a closed proper subset $D''\subsetneq M_{NN'}$ such that for all $f\in M_{NN'}(k)\setminus D ''$ we have the equality of subsets of $V$:
        \[ (\pi _f^{-1}\pi _f (s) \setminus \br s ) 
        \cap B  = \emptyset .\]
    \end{proposition}
    (To extract Proposition \ref{prop:Chow-in-Kai} from {\it loc.\ cit.}, set $X:= V$,
    $Y:= \Spec (k)$, $p=0$, $W:=B$ and $V:=\br s $.
     Also note that the only topological space having dimention $\le -1$ is the empty set.)
    
    Now set $D_2:= D'\cup D'' $ and suppose $f\in M_{NN'}(k)\setminus (D_1\cup D_2)$. Then we know $\pi _f^{-1}\pi _f (s) \setminus \br s  = \br {x_1,\dots ,x_m} $ by the first half of this proof.
    Applying Proposition \ref{prop:Chow-in-Kai} to $B:= \ol B(\Gamma )$
    we get the desired result.
    This completes the proof of Proposition \ref{prop:Chow}. 
\end{proof}

\begin{corollary}\label{cor:Chow}
    Take any $f\in M_{NN'}(k)\setminus (D_1\cup D_2)$ and form the fiber product
    $V\times _{\A^N} U$
    of $\pi _f\colon V\to \A^N$ and $\pi _f|_U \colon U\to \A^N$.
    Then it decomposes as 
    \[ V\times _{\A^N }U \cong U\sqcup \Uprime  \]
    where $\Uprime $ is finite and \'etale over $U$ by the second projection 
    and maps into $V\setminus \ol B(\Gamma )$ along the first projection.
\end{corollary}
\begin{proof}
    Since $\pi _f $ is finite and \'etale over $U_0$ by Proposition \ref{prop:Chow} and $U$ maps into $U_0$, the second projection $V\times _{\A^N} U \to U$ is \'etale (and finite by default because $f\in M_{NN'}(k)\setminus D_1$).
    Since it has the diagonal splitting $U\to V\times _{\A^N}U$ we have $V\times _{\A^N}U\cong U\sqcup \Uprime $. 
    By Proposition \ref{prop:Chow}, the set of its closed points can be computed as $V\times _{\A^N} \br{s}\cong \br{s,x_1,\dots ,x_m}$ with the right hand side having the reduced structure,
    and we know that $x_i$'s map into $V\setminus \ol B(\Gamma )$ by the first projection. This completes the proof.
\end{proof}

Consider the following Cartesian diagram where $\id $ denotes $\id _{X\times \Delta ^1}$:
\begin{equation}\label{eq:cartesian}
    \xymatrix@C=1.5cm{
        \Gamma \ar@{}|{\cap }[d]
        &   
        (U\sqcup \Uprime  )\times X\times \Delta ^1
        \ar@{-}[d]_{\cong }^{\text{Cor.\, \ref{cor:Chow}}}
        \\
        V\times X\times \Delta ^1
        \ar[d]^{\pi _f \times \id }
        &(V\times _{\A ^N}U) \times X\times \Delta ^1
        \ar[l]_{\pr _{1} \times \id }
        \ar[d]^{\pr _{2} \times \id }
        \\
        \A^N \times X\times \Delta ^1
        & 
        U\times X\times \Delta ^1.
        \ar[l]^{(\pi _f|_U )\times \id }
        \ar@{}|{\square }[ul]
        } 
    \end{equation} 
The element $(\pi _f^*\pi _{f*} \Gamma )|_U $ in \eqref{eq:tautologically} is represented by the cycle $((\pi _f|_U)\times \id )^*(\pi _f \times \id )_* \Gamma  $
on $U\times X\times \Delta ^1$.
By a slight abuse of notation let us omit $\id $'s from the notation in what follows; for example the previous expression is shortened as $(\pi _f|_U)^*\pi _{f*} \Gamma $.
By base change formula for flat pullback and proper pushforward of algebraic cycles, we know this equals $\pr _{2*}\pr _1^*\Gamma $.
Via the vertical isomorphism in \eqref{eq:cartesian}, 
if we write $\pr _{iU}$ and $\pr _{i\Uprime }$ ($i=1,2$) for the restrictions of the projections, we get
\begin{equation}\label{eq:U-part-and-U'-part}
    (\pi _f^*\pi _{f*} \Gamma )|_U
    =
    \pr _{2U*}\pr _{1U} ^*\Gamma 
    +
    \pr _{2\Uprime *}\pr _{1\Uprime } ^*\Gamma .
\end{equation} 

We know that $\pr _{1U} \colon U\to V$ is the inclusion map
and $\pr _{2U}\colon U\to U$ is the identity map.
Thus the first term is $\Gamma |_U$.
Therefore we can compute the right hand side in \eqref{eq:tautologically} as 
\begin{equation}\label{eq:U'-part-remains}
    (\Gamma - \pi _f^*\pi _{f*} \Gamma )|_U = -\pr _{2\Uprime *} \pr _{1\Uprime } ^*\Gamma . 
\end{equation} 

By Corollary \ref{cor:Chow} we know $\pr _{1\Uprime }$ maps $\Uprime $ into $V\setminus \ol B(\Gamma ) \subset V$. 
In the resulting commutative diagram:
\[
    \xymatrix@C=2.4cm{
        \Gamma \in C_1(X)(V)
        \ar[r]^{(-)|_{V\setminus \ol B(\Gamma )}}
        &
        C_1(X)(V\setminus \ol B(\Gamma ))
        \ar[r]^{\pr _{1\Uprime }^*}
        &
        C_1(X)(\Uprime )
        \\
        &
        \ol C_1(X)(V\setminus \ol B(\Gamma ))
        \ar[r]^{\pr _{1\Uprime }^*}
        \ar@{}|{\cup }[u]
        &
        \ol C_1(X)(\Uprime ),
        \ar@{}|{\cup }[u]
    }
\]
we know $\Gamma |_{V\setminus \ol B(\Gamma)}$ belongs to the subgroup $ \ol C_1(X)(V\setminus \ol B(\Gamma ))$ by the definition of $B (\Gamma )$.
It follows that $\pr _{1\Uprime }^*\Gamma \in \ol C_1(X)(\Uprime )$.
Since $\pr _{2\Uprime }\colon \Uprime \to U$ is finite, 
we conclude $\pr _{2\Uprime *}\pr _{1\Uprime }^*\Gamma \in \ol C_1(X)(U)$.

Combined with \eqref{eq:tautologically} and \eqref{eq:U'-part-remains} this shows that $\Gamma |_U =0 $ in $H_1(Q_\bullet(X))(U)$.
This completes the proof of Theorem \ref{thm:moving-lemma}.

\section{The unramified cohomology $H_\ur^1(-, W_n\Omega^j_{\log})$}
\label{sec:unram-deRW}

The aim of this short section is to prove Proposition \ref{prop:urcoh-logHW-PI} below. 

Let $X$ be a scheme over $\F_p$. For any integer $n\ge 1$, let $W_n\Omega^{\bullet}_{X}$ denote the de Rham-Witt complex of $X/\F_p$~(cf.\ \cite[I, 1.3]{Ill}). For any morphism of $\F_p$-schemes $f\colon Y\to X$, there exists a natural morphism of complexes of $W_n(\sO_Y)$-modules,
\begin{equation}\label{eq:HW}
f^{-1}W_n\Omega_X^{\bullet}\to W_n\Omega_Y^{\bullet}
\end{equation}
(cf.\ \cite[I, (1.12.3)]{Ill}), which is an isomorphism if $f$ is \'etale~(cf.\ \cite[I, Proposition 1.14]{Ill}).

For any $i\ge 0$, we denote by $W_n\Omega^i_{X,\log}$ the logarithmic Hodge-Witt sheaf of $X$ in the sense of \cite[Definition 2.6]{Shiho}. Namely it is the \'etale sheaf on $X$ defined as the image
\begin{equation*}
W_n\Omega^i_{X,\log}\colon= \mathrm{im}\bigl((\sO_X^{\times})^{\otimes i}\to W_n\Omega^i_X\bigl),
\end{equation*}
of the map $(\sO_X^{\times})^{\otimes i}\to W_n\Omega^i_X\,;\,x_1\otimes\cdots\otimes x_i\mapsto d\log[x_1]\wedge\cdots\wedge d\log [x_i]$, where $[x]\in W_n\sO_X$ is the Teichm\"uller representative of any local section $x\in\sO_X$. 
If $f\colon Y\to X$ is a morphism of $\F_p$-schemes, by the functoriality of the de Rham-Witt complexes~(\ref{eq:HW}), there exists a natural morphism of \'etale sheaves on $Y$, 
\begin{equation}\label{eq:log HW}
f^{-1}W_n\Omega_{X,\log}^i\to W_n\Omega^i_{Y,\log}.
\end{equation}


\begin{proposition}\label{prop:urcoh-logHW-PI}
Fix $n>0$ and $i \ge 0$.
We denote by $H_\ur^1(-, W_n \Omega^i_{\log})$
the Zariski sheaf on $\Sm$ associated to 
\begin{align*}
H^1(-, W_n \Omega^i_{\log}) :
X \mapsto H_\et^1(X, W_n \Omega^i_{X, \log}).
\end{align*}
Then, $H_\ur^1(-, W_n \Omega^i_{\log})$
is a $\P^1$-invariant Nisnevich sheaf,
and has a structure of presheaf with transfers.
\end{proposition}

\begin{remark}\label{rem:use-rec}
As mentioned in the introduction,
it is known that $H_\ur^1(-, W_n \Omega^i_{\log})$
has reciprocity \cite[\S 11.1 (5)]{BRS},
hence it is $\P^1$-invariant by \cite[Theorem 8]{rec}.
We shall give a direct proof of Proposition of \ref{prop:urcoh-logHW-PI} 
which makes no use of the theory of reciprocity sheaves.
\end{remark}

To ease the notation, for $q=0,1$, we put 
\[ F^{q,i} := H^q(-, W_n \Omega^i_{\log}),
\quad
   F^{q,i}_\ur := H^q_\ur(-, W_n \Omega^i_{\log}).
\]
We have 
$F^{q,i}(S)=F^{q,i}_\ur(S)$ for any local $S$.

\begin{theorem}\label{gersten}
For any $X \in \Sm$ and any $q=0,1$, we have an exact sequence
\[ 0 \to F^{q,i}_\ur(X)
\to \bigoplus_{x \in X^{(0)}} F^{q,i}(x)
\to \bigoplus_{x \in X^{(1)}} G^{q,i}_x(X),
\]
where we set $G_x^{q,i}(X):=H^{q+1}_x(X, W_n \Omega^i_{X, \log})$. Moreover, if $q=0$, there exists a canonical isomorphism
\[
\theta^{i}_x\colon F^{0,i-1}(k(x))\xrightarrow{\simeq}G_x^{0,i}(X)
\]
for any codimension one point $x\in X^{(1)}$. 
\end{theorem}

\begin{proof}
For the first assertion, see \cite[Theorem 1.4]{GS}\cite[Theorem 4.1]{Shiho}. For the last assertion, see \cite[Theorem 3.2]{Shiho}. 
\end{proof}

\begin{prop}\label{prop:Omega_log PST}
The \'etale sheaf $W_n\Omega_{\log}^i$ on $\Sm$ has a structure of presheaf with transfers. Hence so does the \'etale cohomology group $H^j(-,W_n\Omega_{\log}^i)$ for any $j\ge 0$.
\end{prop}

\begin{proof}
According to \cite[6.21]{MVW}, the second assertion is immediate from the first one. Therefore, it suffices to show that $W_n\Omega^i_{\log}\in\PST$. However, thanks to Theorem \ref{gersten} together with the theorem of Bloch--Gabber--Kato\,\cite{BK86}, for any $X\in\Sm$, we have a natural exact sequence
\[
0\to H^0(X,W_n\Omega^i_{\log})\to\bigoplus_{x \in X^{(0)}} K^{\rm M}_i(k(x))/p^n
\xrightarrow{(\partial^{\rm M}_x)} \bigoplus_{x \in X^{(1)}}K_{i-1}^{\rm M}(k(x))/p^n,
\]
where $\partial_x^{\rm M}$ is the tame symbol at each $x\in X^{(1)}$. This implies that the sheaf $W_n\Omega^i_{\log}$ has a structure of (homotopy invariant) presheaf with transfers~(cf.\ \cite{K}). This completes the proof of the proposition. 
\end{proof}

\begin{prop}\label{prop:totaro}
Let $f : Y \to X$ be an \'etale morphism in $\Sm$
which induces an isomorphism $k(y) \cong k(x)$
for some $y \in Y^{(1)}$ and $x:=f(y)$.
Then, for $q=0,1$, in the commutative diagram
\[ 
\xymatrix{
F^{q,i}(\Frac \sO_{X, x})/F^{q,i}(\sO_{X, x}) \ar[r] \ar[d] &
F^{q,i}(\Frac \sO_{Y, y})/F^{q,i}(\sO_{Y, y}) \ar[d]
\\
G_x^{q,i}(X) \ar[r] & G_y^{q,i}(Y),
}
\]
all maps are bijective.
\end{prop}

\begin{proof}
For $q=0$, the assertion follows from the last claim of Theorem \ref{gersten}. So, let us assume that $q=1$. Then the bijectivity of the left vertical map is a consequence of the localization sequence
\[ F^{1,i}(\sO_{X, x}) \to F^{1,i}(\Frac \sO_{X, x})
\to G_x^{1,i}(X) \to H^2(\sO_{X, x}, W_n \Omega^i_{X, \log})=0,
\]
and the same for the right vertical map. The statement for the upper horizontal map in the case when $n=1$ is a consequence of \cite[Theorem 4.3]{Totaro}. Indeed, given any discrete valuation ring over $k$,
the cited theorem shows that 
there is an exhaustive filtration 
$F^{1,i}(R)=U_{-1} \subset U_0 \subset U_1 \subset \cdots \subset F^{1,i}(\Frac R)$
whose graded quotients are described solely in terms of the residue field.
In our situation, 
$F^{1,i}(\Frac \sO_{X, x}) \to F^{1,i}(\Frac \sO_{Y, y})$
respects this filtration
because $\sO_{X, x} \to \sO_{Y, y}$ is \'etale.
Hence the desired bijectivity follows
from the assumption $k(y) \cong k(x)$. For $n>1$, thanks to the exact sequence
\[
0\to W_{n-1}\Omega^i_{X,\log}\to W_{n}\Omega^i_{X,\log}\to \Omega^i_{X,\log}\to 0
\]
for any regular scheme $X$ over $\F_p$~(cf.\ \cite[Proposition 2.12]{Shiho}), one can inductively see the bijectivity of the map $F^{1,i}(\Frac \sO_{X, x})/F^{1,i}(\sO_{X, x})\to F^{1,i}(\Frac \sO_{Y, y})/F^{1,i}(\sO_{Y, y})$. This completes the proof. 
\end{proof}

\begin{proof}[Proof of Proposition \ref{prop:urcoh-logHW-PI}]
We first show that $F^{1,i}_\ur$ is a sheaf for Nisnevich topology.
For this, we take a Cartesian diagram in $\Sm$
\[
\xymatrix{
V \ar[r] \ar[d] & Y \ar[d]^f \\
U \ar[r]^j & X,
}
\]
where 
$j$ is an open dense immersion
and $f$ is an \'etale morphism 
that is an isomorphism over $X \setminus j(U)$.
By \cite[12.7]{MVW}, 
it suffices to prove the exactness of the upper row
in the commutative diagram:
\[
\xymatrix{
& 0 \ar[d] & 0 \ar[d] & 0 \ar[d]
\\
0 \ar[r] &
F^{1,i}_\ur(X) \ar[r] \ar[d] &
F^{1,i}_\ur(U) \oplus F^{1,i}_\ur(Y) \ar[r] \ar[d] &
F^{1,i}_\ur(V) \ar[d] 
\\
0 \ar[r] &
\underset{x \in X^{(0)}}{\bigoplus} F^{1,i}(x) \ar[r] \ar[d] &
\underset{x \in U^{(0)}}{\bigoplus} F^{1,i}(x) \oplus
\underset{y \in Y^{(0)}}{\bigoplus} F^{1,i}(y) \ar[r] \ar[d] &
\underset{y \in V^{(0)}}{\bigoplus} F^{1,i}(y) 
\\
&
\underset{x \in X^{(1)}}{\bigoplus} G_x^{1,i}(X) \ar[r]_-{(*)}  &
\underset{x \in U^{(1)}}{\bigoplus} G_x^{1,i}(U) \oplus
\underset{y \in Y^{(1)}}{\bigoplus} G_y^{1,i}(Y). & 
}
\]
The second row is exact for obvious reason
and all columns are exact by Theorem \ref{gersten}.
The map $(*)$ is injective by Proposition \ref{prop:totaro}
(and the assumption on $f$).
Now the claim follows by diagram chasing.
As a consequence, we can find that $F^{1,i}_{\ur}$ is the same as the Nisnevich sheaf associated with $F^{1,i}\in\PST$~(cf.\ Proposition \ref{prop:Omega_log PST}). Therefore, we conclude $F^{1,i}_\ur \in \PST$ by \cite[13.1]{MVW}. 

Finally, to show that $F^{1,i}_{\ur}$ is $\P^1$-invariant, we apply Lemma \ref{lem:p1-inv-field} below.
The condition (1) is due to Izhboldin~\cite{Izhboldin}~
(see also \cite[Theorem 4.4]{Totaro}),
and (2) is a part of Theorem \ref{gersten}.
\end{proof}

\begin{lemma}\label{lem:p1-inv-field}
Let $F \in \PST$.
Suppose (1) and (2) below:
\begin{enumerate}
\item 
For any $K \in \Fld_k^\gm$,
$\sigma_K^* : F(\Spec K) \cong F(\P^1_K)$
is an isomorphism, where
$\sigma_K$ denotes the base change of 
the structure morphism 
$\sigma : \P^1 \to \Spec k$.
\item 
Any open immersion $U \hookrightarrow V$ in $\Sm$
induces an injection $F(V) \to F(U)$.
\end{enumerate}
Then $F$ is $\P^1$-invariant.
\end{lemma}
\begin{proof}
Define $G \in \PST$ by the formula
\begin{equation}\label{eq:p1-summand}
 G(U) := \Coker(\sigma^* : F(U) \to 
\uHom(\Z_\tr(\P^1), F)(U)=F(U \times \P^1))
\qquad (U \in \Sm).
\end{equation}
We have a direct sum decomposition
$\uHom(\Z_\tr(\P^1), F) \cong F \oplus G$
(provided by a $k$-rational point of $\P^1$),
and $G(U)=0$ holds if and only if 
the map in \eqref{eq:p1-summand} is an isomorphism.
By (1), we have $G(\Spec K)=0$ for any $K \in \Fld_k^\gm$.
The property (2) for $F$ implies 
the same property for $\uHom(\Z_\tr(\P^1), F)$ and hence for $G$.
We conclude that $G(U) \hookrightarrow G(k(U))=0$
for any (irreducible) $U \in \Sm$,
which means $F$ is $\P^1$-invariant.
\end{proof}

%


\appendix

\section{proof of Remark \ref{prop:blowup}}\label{sec:app}

We give a proof of Remark \ref{prop:blowup}.
(1) is obvious.
(2) is a consequence of 
the formula $h_0(X \times X')=h_0(X) \otimes h_0(X')$.
For (3) and (4), we shall freely use 
Voevodsky's triangulated category $\DM^-_\eff(k) \subset D^-(\NST)$ of
effective motivic complexes over $k$ \cite{voetri}.
For $V \in \Sm$,
we denote the motivic complex of $V$ by
$M(V):=C_*(\Z_\tr(V)) \in \DM^-_\eff(k)$.
Recall that its homology sheaves 
$h_n(V):=H_n(M(V))$ are trivial if $n<0$
and for $n=0$ it recovers $h_0(V)$ defined in (V\ref{V2}).

\begin{lemma}\label{lem:app}
Let $c>0$ and let $j : U \to X$ be an open immersion in $\Sm$.
Define $M(X/U)$ to be the cone of $j_*:M(U) \to M(X)$.
If each component of 
$Z := (X \setminus U)_\red$ is of codimension $\ge c$ in $X$,
then $h_n(X/U):=H_n(M(X/U))$ vanishes for any $n<c$.
\end{lemma}
\begin{proof}
If $Z$ is smooth over $k$,
then by the Gysin triangle \cite[Thm. 15.15]{MVW} we have
\[ M(X/U) \cong M(Z)(c)[2c] \cong M(Z) \otimes (\G_m[0])^{\otimes c}[c],
\]
from which the statement follows.
In general, let $Z' \subset Z$ be the singular locus of $Z$,
and $U':=X \setminus Z'$.
There is a distinguished triangle
\[ M(U'/U) \to M(X/U) \to M(X/U') \to M(U'/U)[1]. \]
Since each component of $Z'$ has codimension $\ge c+1$ in $X$,
we may assume $h_n(X/U')=0$ for any $n \le c$ by induction.
Since $U' \setminus U = Z \setminus Z'$ is smooth over $k$,
we have shown $h_n(U'/U)=0$ for $n<c$.
It follows that $h_n(X/U)=0$ for $n<c$.
\end{proof}

Now (3) immediately follows from Lemma \ref{lem:app}.
To show (4), let $f : X \to Y$ be a proper birational morphism in $\Sm$.
Let $V \subset Y$ be the open dense subset 
on which  $f^{-1}$ is defined.
Then $f$ restricts to an isomorphism 
$U:=f^{-1}(V) \overset{\cong}{\longrightarrow} V$ 
and $Y \setminus V$ has codimension $\ge 2$ in $Y$.
We have a commutative diagram
\[
\xymatrix{
&
h_0(U) \ar[r] \ar[d]_\cong &
h_0(X) \ar[r] \ar[d] 
& h_0(X/U)=0 \ar[d]
\\
0=h_1(Y/V) \ar[r] &
h_0(V) \ar[r]  &
h_0(Y) \ar[r]  
& h_0(Y/V)=0,
}
\]
in which we used Lemma \ref{lem:app} for the vanishing.
This proves (4).
\qed


\bibliographystyle{plain}

\end{document}

%% file: ClosureMovingTikZ2.tex
\begin{figure}

    \begin{tikzpicture}[scale=0.3]

        \coordinate [label=above:$u_0$] (0) at (-10,6.92);
        \coordinate [label=left:$u_1$] (1) at (-14,0); 
        \coordinate [label=right:$u_2$] (2) at (-6,0); 
        \draw [thick] (0) -- (1) -- (2) -- (0); 
        \node at (-14,6) {$\Delta ^2$}; 

        \def\ippen{5} 
        \def\fringe{1} 

        \coordinate (00) at (2,1); 
        \coordinate (01) at ($(00)+(0,\ippen )$      ); 
        \coordinate (10) at ($(00)+(\ippen ,0)$      ); 
        \coordinate (11) at ($(00)+(\ippen ,\ippen )$); 

        \draw[thick] ($(00)-(0,1) $)-- ($(01)+(0,1) $); 
        \draw[thick] ($(10)-(0,1) $)-- ($(11)+(0,1) $); 
        \draw[thick] ($(10)-(0,1)+(3 ,0) $)-- ($(11)+(0,1)+(3 ,0) $); 
        \node [fill=black,inner sep=1pt, label=right:0] at ($(10)+(3 ,0)$) {};
        \node [fill=black,inner sep=1pt, label=right:1] at ($(11)+(3 ,0)$) {};
        \node at ($(11)+(0,2)+(3  ,0)$) {$\A ^1$}; 
        
        \draw[thick] ($(00)-(1,0) $)-- ($(10)+(1,0) $); 
        \draw[thick] ($(01)-(1,0) $)-- ($(11)+(1,0) $); 
        \draw[thick] ($(00)-(1,0)-(0,3  ) $)-- ($(10)+(1,0)-(0,3  ) $); 
        \node [fill=black,inner sep=1pt, label=above:$v_0$] at ($(00)-(0,3  )$) {};
        \node [fill=black,inner sep=1pt, label=above:$v_1$] at ($(10)-(0,3  )$) {};
        \node at ($(00)-(2,0)-(0,3  )$) {$\Delta^1$}; 

        \draw [thick] ($(00)-(1,1)$)--($(11)+(1,1)$); 
        \node at ($(10)+(-1.5,1.5)$) {$\sigma _1$}; 
        \node at ($(01)+(1.5,-1.5)$) {$\sigma _2$}; 

        \draw [thick,->] (-6,2) -- (0,1); 
        \draw [thick,->] (-6,5) -- (0,\ippen ); 

     \end{tikzpicture}
    \caption{triangulation}\label{fig:triangulation}

    \end{figure}
    